\let\savedVert\|

\makeatletter
\def\cup@reference@code{%
  \RequirePackage[style=numeric,sorting=none,backend=biber]{biblatex}%
  \renewcommand*{\bibfont}{\footnotesize}%
}
\makeatother

\documentclass[
  journal=largetwo,
  manuscript=article,
  manuscriptlabel={Research Preprint},
  logo=false,
  simfonts=false,
  year=2026,
  volume=1,
]{cup-journal}

\let\|\savedVert

\usepackage[utf8]{inputenc}
\usepackage[T1]{fontenc}
\usepackage{amsmath,amssymb,amsthm}
\usepackage{mathtools}
\usepackage[varg]{newtxmath}
\DeclareMathSizes{10.95}{9.95}{7}{5}
\usepackage{graphicx}
\usepackage{booktabs}
\usepackage{enumitem}
\setlist{nosep,leftmargin=*}
\usepackage{algorithm}
\usepackage{algpseudocode}
\usepackage{xcolor}
\usepackage{microtype}
\usepackage[colorlinks=true,linkcolor=blue,citecolor=blue]{hyperref}
\usepackage{cleveref}
\usepackage{orcidlink}

\usepackage{etoolbox}
\AtBeginEnvironment{table}{\footnotesize\setlength{\tabcolsep}{3pt}}
\AtBeginEnvironment{align*}{\small}

\definecolor{primary}{HTML}{006BA2}
\definecolor{accent}{HTML}{E3120B}
\definecolor{neutral}{HTML}{4A4A4A}

\definecolor{plotblue}{HTML}{006BA2}
\definecolor{plotred}{HTML}{E3120B}
\definecolor{plotgreen}{HTML}{2E7D32}
\definecolor{plotorange}{HTML}{FF6F00}

\definecolor{boxfill}{HTML}{E8F4FD}
\definecolor{boxborder}{HTML}{006BA2}

\definecolor{hlblue}{HTML}{E3F2FD}
\definecolor{hlred}{HTML}{FFEBEE}

\newtheorem{theorem}{Theorem}[section]
\newtheorem{lemma}[theorem]{Lemma}

\newtheorem{corollary}[theorem]{Corollary}

\theoremstyle{definition}
\newtheorem{definition}[theorem]{Definition}

\theoremstyle{remark}
\newtheorem{remark}[theorem]{Remark}

\newcommand{\paperbibliography}{\printbibliography}

\makeatletter
\def\cup@journal@name{Fast Evaluation of Truncated Neumann Series}
\def\cup@manuscript{preprint}
\def\cup@year{February 2026}
\makeatother

\title{Fast Evaluation of Truncated Neumann Series\\
  by Low-Product Radix Kernels}
\author{Piyush Sao\,\orcidlink{0000-0002-9432-5855}}
\affiliation{Oak Ridge National Laboratory, Oak Ridge, TN, USA}
\email{saopk@ornl.gov}
\makeatletter\def\cup@contact@details{}\makeatother

\begin{document}

\begin{abstract}
%% Abstract - OCAR narrative

%% Opening + Challenge
Truncated Neumann series $S_k(A)=I+A+\cdots+A^{k-1}$ are used in approximate
matrix inversion and polynomial preconditioning.
In dense settings, matrix-matrix products dominate the cost of evaluating $S_k$.
Naive evaluation needs $k-1$ products, while splitting methods reduce this to
$O(\log k)$.
Repeated squaring, for example, uses $2\log_2 k$ products, so further gains
require higher-radix kernels that extend the series by $m$ terms per update.
Beyond the known radix-5 kernel, explicit higher-radix constructions were not
available, and the existence of exact rational kernels was unclear.

%% Action + Resolution
We construct \emph{radix kernels} for $T_m(B)=I+B+\cdots+B^{m-1}$ and use them
to build faster series algorithms.
For radix 9, we derive an exact 3-product kernel with rational coefficients,
which is the first exact construction beyond radix 5.
This kernel yields $5\log_9 k=1.58\log_2 k$ products, a 21\% reduction from
repeated squaring.
For radix 15, numerical optimization yields a 4-product kernel that matches the
target through degree 14 but has nonzero \emph{spillover} (extra terms) at
degrees $\ge 15$.
Because spillover breaks the standard telescoping update, we introduce a
residual-based radix-kernel framework that accommodates approximate kernels and
retains coefficient $(\mu_m+2)/\log_2 m$.
Within this framework, radix 15 attains $6/\log_2 15\approx 1.54$, the best
known asymptotic rate.
Numerical experiments support the predicted product-count savings and associated
runtime trends.

\end{abstract}

{%
\renewcommand{\thefootnote}{\fnsymbol{footnote}}%
\footnotetext{\noindent\rule{0.8\columnwidth}{0.4pt}\\[2pt]This manuscript has been authored by UT-Battelle,
  LLC under Contract No.\ DE-AC05-00OR22725 with the
  U.S.\ Department of Energy. The publisher, by accepting the
  article for publication, acknowledges that the United States
  Government retains a non-exclusive, paid-up, irrevocable,
  world-wide license to publish or reproduce the published form
  of this manuscript, or allow others to do so, for United States
  Government purposes. The Department of Energy will provide
  public access to these results of federally sponsored research
  in accordance with the DOE Public Access Plan
  (\url{http://energy.gov/downloads/doe-public-access-plan}).}%
}

%% Shared body - included by main.tex and arxiv.tex
%% Each wrapper defines \paperbibliography before \input{body}

%% Main content
%% Introduction - OCAR with forward flow

\section{Introduction}
\label{sec:introduction}

%% OPENING: What is the object and why it matters
The truncated Neumann series is a core object in numerical linear algebra.
For a matrix $A$ with spectral radius $\rho(A) < 1$, it is defined by
\begin{equation}
  \label{eq:neumann}
  S_k(A) = \sum_{j=0}^{k-1} A^j = I + A + A^2 + \cdots + A^{k-1}.
\end{equation}
As $k \to \infty$, $S_k(A)$ converges to $(I-A)^{-1}$, so finite truncations give
practical approximate inverses.
These approximations appear in polynomial preconditioning~\cite{dubois1979,saad2003},
resolvent approximation~\cite{higham2008}, log-determinant estimation~\cite{sao2025logdet},
and massive MIMO signal detection~\cite{gustafsson2017}.
In dense problems, the dominant cost is matrix-matrix multiplication (GEMM), so
reducing product count is the main efficiency goal.

Existing splitting methods already reduce the naive $k-1$ products to
$O(\log k)$~\cite{lei1992}.
Binary splitting (repeated squaring) uses $2\log_2 k$ products.
Hybrid radix-$\{2,3\}$ schemes~\cite{dimitrov1995} lower the constant only
slightly, to about $1.9\log_2 k$.
Thus asymptotic order is no longer the bottleneck; the leading coefficient is.

%% CHALLENGE: The gap we address
Improving that coefficient requires better radix kernels.
Let $T_m(B)=I+B+\cdots+B^{m-1}$ and let $\mu_m$ denote the minimum number of
products to evaluate $T_m$.
Using the decomposition
\begin{equation}
  \label{eq:decomposition}
  S_{mn}(A) = S_n(A) \cdot T_m(A^n).
\end{equation}
This identity gives one radix-$m$ update cost as
$C(m)=\mu_m+2$ products: $\mu_m$ for the kernel, one
for concatenation, and one for updating the next power.
Hence the asymptotic coefficient is $C(m)/\log_2 m$.
Each product can at most double polynomial degree, so
$\mu_m \ge \lceil\log_2(m-1)\rceil$.
This lower bound gives $\mu_9 \ge 3$ and $\mu_{15} \ge 4$, which suggests that
higher-radix kernels could improve constants if such kernels exist.
Gustafsson et al.~\cite{gustafsson2017} achieved $\mu_5 = 2$ for radix 5, but
explicit constructions beyond radix 5 were missing.
This leaves two linked questions: can radix 9 reach its lower bound exactly, and
can higher-radix kernels remain useful when they are only approximate?

%% ACTION: What we contribute to fill the gap
We answer these questions with three contributions.
First, we construct an exact radix-9 kernel with $\mu_9 = 3$ and rational
coefficients (Theorem~\ref{thm:radix9}), giving
$5\log_9 k = 1.58\log_2 k$ products.
This is the first exact construction beyond radix 5 and gives a 21\% improvement
over binary splitting.
Second, for radix 15 we obtain a 4-product kernel by numerical optimization; it
matches the target prefix through degree 14 but has nonzero spillover
(unwanted higher-degree terms) at degrees $\ge 15$.
Third, we develop a residual-based general radix-kernel framework that
accommodates spillover and preserves coefficient $(\mu_m+2)/\log_2 m$.
Within this framework, radix 15 reaches $6/\log_2 15 \approx 1.54$, the best
known asymptotic coefficient.

%% RESOLUTION: Contributions and paper map
The paper proceeds as follows.
Section~\ref{sec:preliminaries} reviews background and prior work.
Section~\ref{sec:kernel} presents the radix-9 construction.
Section~\ref{sec:radix15} gives the approximate radix-15 kernel.
Section~\ref{sec:generalradix} develops the general framework, and
Section~\ref{sec:numerics} reports numerical results.

%% Background - Existing methods and their limitations

\section{Background}
\label{sec:preliminaries}

%% OPENING: Review existing methods, preview bottleneck
We review existing methods and show that naive kernel computation is their
main limitation.
Given $A \in \mathbb{R}^{d \times d}$ with $\rho(A) < 1$, define the
truncated Neumann series as in~\eqref{eq:neumann}:
\begin{equation}
  \label{eq:series}
  S_k(A) = \sum_{j=0}^{k-1} A^j = I + A + A^2 + \cdots + A^{k-1}.
\end{equation}
We measure cost by counting matrix products (GEMMs), which dominate
when $d \gtrsim 100$ for dense matrices.

\subsection{Splitting Methods}

Several radix-based schemes exist, differing in how many terms they
add per step.
\begin{itemize}
  \item \textbf{Naive evaluation.}
    The naive approach accumulates powers $I, A, A^2, \ldots$, requiring
    $k-1$ matrix products---prohibitive for large $k$.
  \item \textbf{Radix-$m$ splitting.}
    Faster algorithms use the factorization $S_{mn}(A) = S_n(A) \cdot
    (I + A^n + \cdots + A^{(m-1)n})$.
    A radix-$m$ scheme iterates this identity. Each step computes:
    \begin{itemize}
      \item powers $A^{2n}, \ldots, A^{(m-1)n}$ from $A^n$ ($m-2$ products),
      \item the product $S_n \cdot (I + A^n + \cdots + A^{(m-1)n})$
        (1 product),
      \item and $A^{mn} = A^{(m-1)n} \cdot A^n$ (1 product).
    \end{itemize}
    The naive cost per step is $C(m) = m$ products, giving coefficient
    $m/\log_2 m$.
  \item \textbf{Binary splitting.}
    For $m = 2$, the identity simplifies to $S_{2n} = S_n(I + A^n)$, using
    only 2 products per step~\cite{westreich1989,lei1992}:
    (i)~$A^{2n} = A^n \cdot A^n$;
    (ii)~$S_{2n} = S_n \cdot (I + A^n)$.
    Total: $2\log_2 k$ products.
  \item \textbf{Ternary splitting.}
    For $m = 3$, $S_{3n} = S_n(I + A^n + A^{2n})$:
    (i)~$A^{2n} = A^n \cdot A^n$;
    (ii)~$A^{3n} = A^{2n} \cdot A^n$;
    (iii)~$S_{3n} = S_n \cdot (I + A^n + A^{2n})$.
    Total: $3\log_3 k \approx 1.89\log_2 k$ products.
\end{itemize}

%% Analysis: optimal radix
\paragraph{Optimal radix for naive splitting.}
The coefficient $C(m)/\log_2 m$ equals $m/\log_2 m$, minimized
at $m = e \approx 2.718$; thus $m = 3$ is optimal among integers:
\begin{center}
\begin{tabular}{@{}ccc@{}}
  \toprule
  Radix $m$ & Cost $C(m)$ & Coefficient $C(m)/\log_2 m$ \\
  \midrule
  2 & 2 & $2/\log_2 2 = 2.00$ \\
  3 & 3 & $3/\log_2 3 \approx 1.89$ \\
  4 & 4 & $4/\log_2 4 = 2.00$ \\
  \bottomrule
\end{tabular}
\end{center}
Thus ternary has the best coefficient among pure radix methods.
Dimitrov and Cooklev~\cite{dimitrov1995} extended this to hybrid
radix-$\{2,3\}$ methods for arbitrary $k$, achieving coefficient
$\approx 1.9$.
These mixed $\{2,3\}$ strategies are near-optimal, but only with
naive kernel computation.

\subsection{Radix Kernels}

The splitting methods above compute the partial sum
$T_m(B) = I + B + \cdots + B^{m-1}$ naively, requiring $m - 2$ products
for powers $B^2, \ldots, B^{m-1}$.
We call $T_m$ the \emph{kernel}.
If the kernel cost $\mu_m$ grew only logarithmically in $m$, the
coefficient $(\mu_m + 2)/\log_2 m$ could fall below $1.89$.
Two identities make this possible:
\begin{itemize}
  \item \textbf{Block concatenation.}
    $S_{mn}(A) = S_n(A) \cdot T_m(A^n)$ for any $m, n \geq 1$,
    verified by expanding both sums.
  \item \textbf{Telescoping.}
    $T_m(B)(I - B) = I - B^m$\label{eq:telescope}
    (the geometric series identity), which extracts $B^m$ from the
    kernel without extra products.
\end{itemize}

\paragraph{Cost model.}
Let $\mu_m$ denote the minimum number of products to compute $T_m(B)$.
A radix-$m$ update $(S_n, A^n) \to (S_{mn}, A^{mn})$ uses:
\begin{itemize}
  \item $\mu_m$ products for the kernel,
  \item 1 product for telescoping:
    $I - A^{mn} = T_m(A^n)(I - A^n)$,
  \item 1 product for concatenation:
    $S_{mn} = S_n \cdot T_m(A^n)$.
\end{itemize}
The total is $C(m) = \mu_m + 2$, giving coefficient $C(m)/\log_2 m$.

\paragraph{Lower bound.}
Since each product at most doubles polynomial degree, $p$ products
starting from $\{I, B\}$ reach degree at most $2^p$.
Because $T_m$ has degree $m-1$, evaluating it requires at least
$\lceil \log_2(m-1) \rceil$ products.
In particular, $\mu_5 \geq 2$ and $\mu_9 \geq 3$.
Can these bounds be achieved?

\subsection{Prior Work: Quinary (Radix-5) Kernel}
\label{subsec:quinary}

For radix-5, the answer is yes.
Gustafsson et al.~\cite{gustafsson2017} achieved $\mu_5 = 2$ with
the following construction:
(i)~$U := B^2$;
(ii)~$V := U(B + U) = B^3 + B^4$;
(iii)~$T_5(B) := I + B + U + V$.
Direct verification confirms $T_5(B) = I + B + B^2 + B^3 + B^4$.

With $\mu_5 = 2$, a quinary update costs $C(5) = 4$ products,
giving coefficient $4/\log_2 5 \approx 1.72$.

%% RESOLUTION: Establish unknowns and challenges
\paragraph{Open questions.}
The quinary kernel reduces the coefficient, but gaps remain.
First, no radix-9 construction exists, though
the degree-doubling bound permits $\mu_9 = 3$, giving coefficient
$5/\log_2 9 \approx 1.58$.
Second, it is unknown whether exact rational kernels exist for higher
radices such as $m = 15$. Approximate constructions may have nonzero
\emph{spillover} into degrees $\geq m$, breaking the telescoping
identity~\eqref{eq:telescope} and preventing exact extraction
of $B^m$. No framework exists for such kernels in series algorithms.
Sections~\ref{sec:kernel}--\ref{sec:generalradix} address these gaps.

%% Radix-9 Kernel

\section{Radix-9 Kernel}
\label{sec:kernel}

%% OPENING: Motivation
The quinary kernel~\cite{gustafsson2017} achieves coefficient $4/\log_2 5 \approx 1.72$.
Can we improve on this by going to higher radices?
For radix-9, beating $1.72$ requires $C(9)/\log_2 9 < 1.72$, i.e., at most
5 products. Since $C(9) = \mu_9 + 2$, we need $\mu_9 \leq 3$. The
degree-doubling bound gives $\mu_9 \geq 3$, so 3~products is optimal.

\paragraph{Derivation strategy.}
To achieve 3 products, we seek a factorization $P \cdot Q$ where $P, Q$ are linear combinations of $\{B, U, V\}$ with $U = B^2$ and $V = B^3 + 2B^4$ (computable in 2~products).
The product $P \cdot Q$ spans degrees 2--8. We choose coefficients so that:
(i) degrees 5--8 have coefficient~1 directly from convolution; (ii) degrees
2--4 are corrected by post-adding multiples of $U$ and $V$. This yields 8
equations in 8 unknowns (4 coefficients each for $P$ and $Q$), which we solve
symbolically.

%% The radix-9 theorem
\begin{theorem}[Radix-9 kernel in 3 products]
  \label{thm:radix9}
  The kernel $T_9(B) = I + B + \cdots + B^8$ is computed by:
  \begin{enumerate}
    \item $U := B^2$ \hfill (MM1)
    \item $V := U(B + 2U) = B^3 + 2B^4$ \hfill (MM2)
    \item $P := \tfrac{3}{20}B + 2U + V$, \quad $Q := \tfrac{11}{40}B - \tfrac{1}{8}U + \tfrac{1}{4}V$
    \item $W := P \cdot Q$ \hfill (MM3)
    \item $T_9(B) := W + I + B + \tfrac{767}{800}U + \tfrac{15}{32}V$
  \end{enumerate}
\end{theorem}

The key insight is that one carefully chosen product $P \cdot Q$ simultaneously
generates four high-degree terms ($B^5$ through $B^8$), while the
low-degree corrections adjust $B^2$--$B^4$ using only additions.

\begin{proof}
  Let $\mathbf{p} = [0.15, 2, 1, 2]$ and $\mathbf{q} = [0.275, -0.125, 0.25, 0.5]$
  be the coefficients of $P$ and $Q$ in powers $B^1, B^2, B^3, B^4$. Direct
  calculation shows convolution yields coefficient~1 for $B^5$ through $B^8$.
  The low-degree terms $B^2, B^3, B^4$ are corrected to coefficient~1 by
  adding the stated fractions of $U$ and $V$.
\end{proof}

\begin{remark}[Rational coefficients]
  \label{rem:rational}
  The radix-9 kernel uses only rational coefficients (fractions with denominators dividing 800), enabling exact computation in rational or fixed-point arithmetic.
  This contrasts with radix-15 (Section~\ref{sec:radix15}), where numerical optimization yields only real coefficients with nonzero ``spillover'' into higher-degree terms.
  The transition from exact rational solutions (radix-5, radix-9) to approximate real solutions (radix-15) may suggest a fundamental boundary in algebraic kernel construction.
\end{remark}

%% RESOLUTION: What this achieves
Having established the 3-product construction, we now quantify its benefit.
With $\mu_9 = 3$, the update costs $C(9) = 5$ products. Iterating $t = \log_9 k$
times gives $S_k(A)$ in $5\log_9 k \approx 1.58\log_2 k$ products---a 21\%
improvement over binary and 8\% over quinary.

In summary, the radix-9 kernel achieves the best known asymptotic coefficient among exact constructions, at the cost of one additional matrix product per update.
The next section explores whether higher radices can improve further.

%% Approximate Radix-15 Kernel

\section{Approximate Radix-15 Kernel}
\label{sec:radix15}

%% OPENING: Push beyond radix-9
The radix-9 kernel is exact and achieves the degree-doubling bound
$\mu_9 = 3$ with rational coefficients. Can we go further?
For radix-15, the bound gives $\mu_{15} \geq \lceil\log_2 14\rceil = 4$.
A 4-product kernel would yield update cost $C(15) = 6$ and coefficient
$6/\log_2 15 \approx 1.54$---provided the kernel can be constructed.

%% CHALLENGE
Unlike radix-9, where 8 unknowns suffice, radix-15 has 28 free
parameters across 4 chained products---a search space too large for
ad hoc construction, requiring a systematic approach.
Expressing factor matrices as linear combinations over a polynomial
basis and matching coefficients to $T_{15}(B)$ yields a nonlinear
system with no known symbolic solution, so we turn to numerical
optimization.
Even numerically, no exact solution has been found: we therefore
relax the target, allowing nonzero \emph{spillover} into degrees
$\geq 15$.
Whether spillover terms compromise the radix-kernel framework---or
can be absorbed---is deferred to Section~\ref{sec:generalradix}.

\paragraph{Search methodology.}
We build polynomial basis sets $\mathcal{B}_i$ incrementally from
products, express each factor pair $(L_i, R_i)$ as linear combinations
over $\mathcal{B}_i$, and minimize coefficient error using L-BFGS-B.

\subsection{Circuit Structure}

Fix $P_1 = B^2$ (1~product). Define basis sets
$\mathcal{B}_1 = \{I, B, P_1\}$, and for $i = 2, 3, 4$:
\begin{align}
  L_i &= \sum_{X \in \mathcal{B}_i} \ell_{i,X} X, \quad
  R_i = \sum_{X \in \mathcal{B}_i} r_{i,X} X, \label{eq:factors} \\
  P_i &= L_i \cdot R_i, \quad
  \mathcal{B}_{i+1} = \mathcal{B}_i \cup \{P_i\}. \notag
\end{align}
The final kernel is
\begin{equation}
  \label{eq:radix15kernel}
  \widetilde{T}_{15}(B) = I + \gamma_1 B + \gamma_2 P_1
    + \gamma_3 P_2 + \gamma_4 P_3 + P_4.
\end{equation}
This uses exactly 4 matrix products: one for $P_1$, and one each
for $P_2, P_3, P_4$.

\subsection{Optimization and Results}

Let $[P(B)]_j$ denote the coefficient of $B^j$ in polynomial $P(B)$.
We minimize $\sum_{j=0}^{14}([\widetilde{T}_{15}]_j - 1)^2$ over
the 28 free parameters
$(\ell_{i,\cdot}, r_{i,\cdot}, \gamma)$ using L-BFGS-B.
The solver achieves residual $\approx 3 \times 10^{-12}$, yielding
\begin{equation*}
  \max_{0 \leq j \leq 14}
  \bigl| [\widetilde{T}_{15}(B)]_j - 1 \bigr| < 9 \times 10^{-7}.
\end{equation*}

\paragraph{Spillover.}
The construction produces nonzero higher-degree terms:
\begin{equation*}
  \widetilde{T}_{15}(B) = \sum_{j=0}^{14} B^j
    + 0.185\, B^{15} + 0.458\, B^{16} + O(B^{17}).
\end{equation*}
Since $\widetilde{T}_{15}(B) \neq T_{15}(B)$ as polynomials, this
is an \emph{approximate} kernel: $(1-z)\widetilde{T}_{15}(z) = 1 + O(z^{15})$
but not $1 - z^{15}$ exactly.
The leading spillover coefficient
$[\widetilde{T}_{15}]_{15} \approx 0.19$ controls convergence
(formalized in Section~\ref{sec:generalradix}).
The spillover arises because the 28-parameter ansatz cannot
simultaneously satisfy all 15 target constraints while zeroing
higher-degree terms.

\begin{remark}[Non-uniqueness]
  \label{rem:nonunique}
  The optimization landscape has multiple local minima: different
  random initializations yield distinct circuits with comparable
  prefix error but different spillover structure (e.g., $c$ ranging
  from $0.5$ to $0.9$ across runs).
  The specific coefficients reported here and in
  \ref{app:radix15} correspond to the best residual found
  over 200 random starts.
  The theoretical framework of Section~\ref{sec:generalradix}
  applies to any solution satisfying the prefix condition.
\end{remark}

%% RESOLUTION
The 4-product approximate kernel achieves the target coefficient
$6/\log_2 15 \approx 1.54$, but produces spillover that prevents
standard telescoping.
Whether an exact radix-15 kernel exists remains open.
The next section shows how the general radix-kernel framework
handles spillover, recovering the $1.54$ coefficient.

%% General Radix-Kernel Framework

\section{General Radix-Kernel Framework}
\label{sec:generalradix}

%% OPENING
The best known exact kernel is radix-9, with coefficient
$5/\log_2 9 \approx 1.58$.
To beat it we need a higher radix.
A 4-product radix-15 kernel would drop the coefficient to
$6/\log_2 15 \approx 1.54$.
Yet numerical search (Section~\ref{sec:radix15}) shows that exact
radix-15 kernels almost certainly do not exist: they produce nonzero
\emph{spillover} at degrees $\geq 15$.
We therefore need a framework that accepts approximate kernels while
keeping the prefix exact.

%% CHALLENGE
Spillover invalidates the telescoping identity
$T_m(A^n)(I - A^n) = I - A^{mn}$
that normally lets us extract $A^{mn}$.
With spillover the identity fails and the power cannot be recovered.
Hensel lifting ($Y' = Y(2I - MY)$) sidesteps power extraction, but
its coefficient is $2.00$---no better than binary splitting.
The question is how to keep the $1.54$ promise of radix-15 without
exact extraction.

%% ACTION
We resolve both obstacles by applying the kernel to the \emph{residual}
rather than to powers of~$A$, eliminating the need for power extraction.
We develop this in three stages: formalize approximate kernels, analyze
their composition as polynomials, then lift to matrices.

\subsection{Approximate Radix Kernels}

We first make precise what ``approximate kernel'' means and how its
error behaves.

\begin{definition}[Approximate radix-$m$ kernel]
  \label{def:approxkernel}
  A polynomial $f \in \mathbb{R}[z]$ is an \emph{approximate radix-$m$
  kernel} if $(1-z)\,f(z) = 1 + O(z^m)$.
  Equivalently, $f(z) \equiv (1-z)^{-1} \pmod{z^m}$.
  The terms of degree $\geq m$ in $f$ are called \emph{spillover}.
  When the spillover is zero, the kernel is \emph{exact} and
  coincides with $T_m(z)$.
\end{definition}

\begin{definition}[Error map]
  \label{def:errormap}
  Given an approximate radix-$m$ kernel~$f$, define
  $E(z) = 1 - (1 - z)\,f(z)$.
\end{definition}

\begin{lemma}[Error order]
  \label{lem:errororder}
  If $f(z) = \sum a_j z^j$ is an approximate radix-$m$ kernel, then
  \[
    E(z) = (1 - a_m)\,z^m + O(z^{m+1}).
  \]
  In particular, $c = 1 - [f]_m$.
  If $f$ is exact, $a_m = 0$, so $c = 1$ and $E(z) = z^m$.
\end{lemma}

\begin{proof}
  Since $(1-z)f(z) = 1 + O(z^m)$, the degree-$m$ term is
  $(1-z)f(z) = 1 + (a_m - a_{m-1})\,z^m + \cdots = 1 + (a_m - 1)\,z^m + \cdots$,
  so $E(z) = 1 - (1-z)f(z) = (1 - a_m)\,z^m + O(z^{m+1})$.
\end{proof}

\subsection{General Radix-Kernel Summation}

The key idea is to apply the kernel to its own error: $f(E(z))$
corrects the residual left by $f(z)$.

\begin{definition}[General radix-kernel summation]
  \label{def:radixsum}
  Given an approximate radix-$m$ kernel $f$ with error map $E$,
  one \emph{lift} of the general radix-kernel summation produces
  \begin{equation}
    \label{eq:radixsum}
    f(z) \cdot f\bigl(1 - (1-z)f(z)\bigr) = f(z) \cdot f(E(z)).
  \end{equation}
\end{definition}

\begin{lemma}[Composition identity]
  \label{lem:composition}
  For any approximate radix-$m$ kernel $f$ with error map $E$:
  \begin{equation}
    \label{eq:composition}
    E(E(z)) = 1 - (1-z)\,f(z)\,f(E(z)).
  \end{equation}
\end{lemma}

\begin{proof}
  By definition, $1 - E(z) = (1-z)\,f(z)$. Applying the definition
  of $E$ at the point $E(z)$:
  \[
    E(E(z)) = 1 - (1-E(z))\,f(E(z)) = 1 - (1-z)\,f(z)\,f(E(z)).
    \qedhere
  \]
\end{proof}

\begin{lemma}[Prefix growth]
  \label{lem:prefixgrowth}
  If $f$ is an approximate radix-$m$ kernel, then
  $f(z) \cdot f(E(z))$ matches the first $m^2$ coefficients of
  $(1-z)^{-1}$.
\end{lemma}

\begin{proof}
  By Lemma~\ref{lem:composition},
  $(1-z)\,f(z)\,f(E(z)) = 1 - E(E(z))$.
  By Lemma~\ref{lem:errororder}, $E(z) = O(z^m)$, so
  $E(E(z)) = O(z^{m^2})$.
  Hence $(1-z)\,f(z)\,f(E(z)) = 1 - O(z^{m^2})$, and
  $f(z)\,f(E(z))$ agrees with $(1-z)^{-1}$ to order $m^2 - 1$.
\end{proof}

\begin{lemma}[Error propagation]
  \label{lem:errorprop}
  Write $E^{[n]}$ for the $n$-fold composition
  ($E^{[0]}(z) = z$, $E^{[n+1]} = E \circ E^{[n]}$).
  Let $E(z) = c\,z^m + O(z^{m+1})$. Then
  \[
    E^{[n]}(z) = c^{(m^n - 1)/(m-1)}\,z^{m^n} + O(z^{m^n+1}).
  \]
  Spillover is harmless to the prefix: regardless of $c$, it is
  pushed to degree $\geq m^n$ after $n$ steps.
\end{lemma}

\begin{proof}
  $E(E(z)) = c\,(cz^m)^m + \cdots = c^{m+1}\,z^{m^2} + \cdots$.
  By induction, $E^{[n]}(z) =
  c^{1+m+\cdots+m^{n-1}}\,z^{m^n} + \cdots = c^{(m^n-1)/(m-1)}\,z^{m^n}
  + \cdots$.
\end{proof}

\begin{remark}[Analytic stability]
  \label{rem:stability}
  In floating-point arithmetic, the coefficient magnitude matters.
  When $|c| < 1$, the leading coefficient $c^{(m^n-1)/(m-1)}$ decays,
  bounding the residual norm.
  For the radix-15 kernel (Section~\ref{sec:radix15}),
  $[f]_{15} \approx 0.19$ gives $c \approx 0.81 < 1$.
\end{remark}

\subsection{Matrix Iteration}

We now translate this polynomial machinery to matrices.
Let $M = I - A$. We maintain an approximate inverse
$Y_n \approx M^{-1}$ and its residual $R_n = I - MY_n$.
Starting from $Y_0 = I$ and $R_0 = A$, each step applies $f$
to the residual:
\begin{equation}
  \label{eq:generalradix}
  Y_{n+1} = Y_n \cdot f(R_n), \qquad R_{n+1} = I - M\,Y_{n+1}.
\end{equation}

\begin{lemma}[General radix-kernel iteration]
  \label{lem:generaliter}
  Under iteration~\eqref{eq:generalradix}:
  \begin{enumerate}
    \item $R_n = E^{[n]}(A)$;
    \item $R_n = O(A^{m^n})$;
    \item $Y_n = \sum_{j=0}^{m^n - 1} A^j + O(A^{m^n})$.
  \end{enumerate}
\end{lemma}

\begin{proof}
  \emph{Base case.} $R_0 = A = E^{[0]}(A)$.
  \emph{Induction.} If $R_n = E^{[n]}(A)$, then
  $R_{n+1} = I - M\,Y_n\,f(R_n) = I - (I - R_n)\,f(R_n)
   = E(R_n) = E^{[n+1]}(A)$.
  Parts~(2)--(3) follow from $E(z) = O(z^m)$ and
  $M\,Y_n = I - R_n = I - O(A^{m^n})$.
\end{proof}

%% Cost Analysis and Comparison (continuation of Section 5)

\subsection{Cost Analysis}
\label{sec:comparison}

We now quantify the cost and compare against prior methods.
Each step~\eqref{eq:generalradix} requires $\mu_m$ products for
$f(R_n)$, $1$ for $Y_{n+1} = Y_n \cdot f(R_n)$, and $1$ for
$R_{n+1} = I - M\,Y_{n+1}$: total $\mu_m + 2$ per
step---identical to the exact framework.
Achieving accuracy $k = m^t$ requires $t = \log_m k$ steps:
\begin{equation}
  \label{eq:generalcost}
  (\mu_m + 2)\log_m k
  = \frac{\mu_m + 2}{\log_2 m}\,\log_2 k.
\end{equation}

\begin{corollary}[Radix-15 coefficient]
  \label{cor:radix15coeff}
  If there exists an approximate radix-15 kernel with $\mu_{15} = 4$
  satisfying $(1-z)f(z) = 1 + O(z^{15})$, then the iteration
  achieves coefficient $6/\log_2 15 \approx 1.54$.
\end{corollary}

Our numerically optimized kernel (\ref{app:radix15}) satisfies
this condition approximately, with prefix error
$\max_{0 \leq j \leq 14}|[f]_j - 1| < 9 \times 10^{-7}$.
Numerical experiments (Section~\ref{sec:numerics}) confirm that it
exhibits the predicted $m^n$ prefix growth.

Table~\ref{tab:compare} compares all methods against prior work.
Prior methods were limited by naive kernel computation, making higher
radices counterproductive and leaving coefficients at or above $2\log_2 k$.

\begin{table}[t]
  \centering
  \caption{Asymptotic product counts for $S_k(A)$ evaluation}
  \label{tab:compare}
  \begin{tabular}{@{}lccc@{}}
    \toprule
    Method & Products & Coeff.\ in $\log_2 k$ & Explicit? \\
    \midrule
    Westreich~\cite{westreich1989} & $3\log_2 k$ & $3.00$ & Yes \\
    Lei--Nakamura~\cite{lei1992} & $2\log_2 k - 1$ & $2.00$ & Yes \\
    Dimitrov--Cooklev~\cite{dimitrov1995} & hybrid & $\approx 1.9$ & Yes \\
    Matsumoto et al.~\cite{matsumoto2018} & higher radix & $< 2.0$ & No \\
    \midrule
    Quinary~\cite{gustafsson2017} & $4\log_5 k$ & $1.72$ & Yes \\
    Radix-9 (this paper) & $5\log_9 k$ & $1.58$ & Yes (exact) \\
    Radix-15 (this paper) & $6\log_{15} k$ & $1.54$ & Yes (approx.) \\
    \bottomrule
  \end{tabular}
\end{table}

The radix-9 scheme achieves 21\% fewer products than
Lei--Nakamura on powers of~9.
The gap widens as $k$ grows: for $k = 9^{10} \approx 3.5 \times 10^9$,
radix-9 requires 50 products versus 63 for binary splitting,
a saving of 13~GEMMs.

The optimization principle: choose $m$ to minimize $C(m)/\log_2 m$.
For our kernels: $4/\log_2 5 \approx 1.72$ (quinary),
$5/\log_2 9 \approx 1.58$ (radix-9),
$6/\log_2 15 \approx 1.54$ (radix-15).
Higher radices help only if kernel cost grows sublinearly in $\log_2 m$.

\subsection{Special Cases}

The framework recovers known methods as special cases.

\paragraph{Hensel lifting ($m = 2$).}
$f(z) = 1 + z$, $E(z) = z^2$, $\mu_2 = 0$:
$Y_{n+1} = Y_n(I + R_n) = Y_n(2I - MY_n)$,
coefficient $2/\log_2 2 = 2.00$.

\paragraph{Exact radix-$m$.}
$f = T_m$, $E(z) = z^m$, $R_n = A^{m^n}$: recovers the standard
telescoping update.

%% RESOLUTION
The framework resolves both obstacles by applying the kernel to the
residual, not to powers of $A$.
The cost per step is $\mu_m + 2$---identical to the exact case---and
the prefix grows as $m^n$ regardless of spillover.
Hensel lifting and exact telescoping emerge as special cases, so the
framework subsumes all prior methods.
With radix-15 it achieves coefficient $1.54$, the best known rate.

%% Numerical Experiments - Compact version

\section{Numerical Experiments}
\label{sec:numerics}

%% OPENING: What we validate and why
Theoretical product counts are meaningless if numerical errors accumulate.
We validate the proposed radix-kernel methods on approximate inverse computation.
All computations use Python (NumPy/OpenBLAS) with IEEE 754 double precision arithmetic.
To isolate the effect of algorithm choice, we use a standard random matrix model.

\paragraph{Setup.}
We generate random matrices $A \in \mathbb{R}^{d \times d}$ with $\rho(A) < 1$ by
constructing $A = \alpha \Omega D \Omega^T$ where $\Omega$ is an orthogonal matrix, $D$ is diagonal with
entries in $[-1, 1]$, and $\alpha < 1$. For each $k = m^t$ (a pure power), we
compare: (i) Lei--Nakamura binary splitting using $S_{2n} = S_n(I + A^n)$ with
2~products per doubling, giving $2\lceil\log_2 k\rceil$ total; (ii) the quinary update~\cite{gustafsson2017};
(iii) our radix-9 update (Theorem~\ref{thm:radix9}); and (iv) our radix-15 approximate update
via the general framework (Section~\ref{sec:generalradix}).
All methods are restricted to $k = m^t$ for fair comparison at pure powers.
We measure accuracy by the Frobenius-norm residual $\|I - (I-A)S_k\|_F$, where $\|\cdot\|_F = (\sum_{ij} a_{ij}^2)^{1/2}$, and report the condition number $\kappa(\cdot)$ as the ratio of largest to smallest singular values.

\paragraph{Results.}
Table~\ref{tab:results} shows matrix product counts for $d = 500$ and $\alpha = 0.9$.
Radix-9 uses 25\% fewer products than binary for $k = 729$,
and radix-15 achieves the same 25\% saving at $k = 225$.
Fewer products means fewer rounding errors accumulate---but does the more complex
kernel structure introduce new error sources?

\begin{table}[t]
  \centering
  \caption{Product counts for $S_k(A)$ at pure powers (binary baseline)}
  \label{tab:results}
  \begin{tabular}{@{}lccccc@{}}
    \toprule
    $k$ & Binary & Quinary & Radix-9 & Radix-15 & Savings \\
    \midrule
    $125 = 5^3$ & 14 & \textbf{12} & --- & --- & 14\% \\
    $625 = 5^4$ & 20 & \textbf{16} & --- & --- & 20\% \\
    $81 = 9^2$ & 14 & --- & \textbf{10} & --- & 29\% \\
    $729 = 9^3$ & 20 & --- & \textbf{15} & --- & 25\% \\
    $225 = 15^2$ & 16 & --- & --- & \textbf{12} & 25\% \\
    $3375 = 15^3$ & 24 & --- & --- & \textbf{18} & 25\% \\
    \bottomrule
  \end{tabular}
\end{table}

%% Residuals and convergence
The exact methods (binary, quinary, radix-9) achieve comparable residuals
($\approx 10^{-14}$) for a given $k$, confirming exact algebraic identities.
Radix-15, using an approximate kernel with prefix error $< 9 \times 10^{-7}$,
converges to a residual floor of $\approx 10^{-6}$ rather than machine precision;
this is expected since the kernel coefficient error limits achievable accuracy.
Figure~\ref{fig:convergence} shows the normalized residual
$\|I - (I-A)S_k\|_F / \sqrt{d}$ versus matrix product count ($d=64$,
$\kappa(I-A) = 10^{4}$, log-spaced eigenvalues).
Higher-radix methods achieve a given accuracy with fewer products.

\begin{figure}[t]
  \centering
  \includegraphics[width=0.85\linewidth]{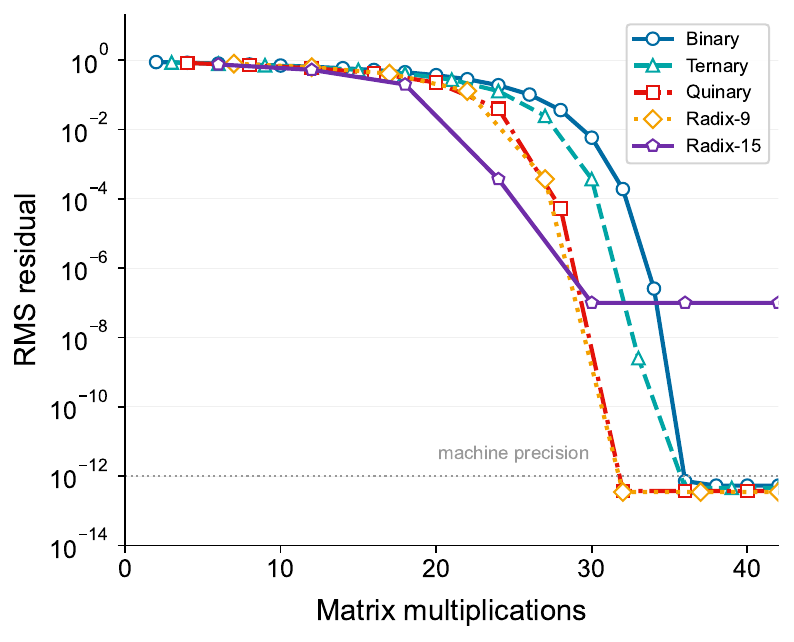}
  \caption{Convergence comparison ($d=64$, $\kappa(I-A) = 10^{4}$, log-spaced
    eigenvalues). Higher-radix methods reach machine precision with fewer matrix products.}
  \label{fig:convergence}
\end{figure}

Beyond accuracy, we examine computational efficiency (asymptotic coefficients are compared in Figure~\ref{fig:coefficients}).
Table~\ref{tab:timing} confirms that product-count reductions translate to
wall-clock savings.
We also test on nonnormal matrices (upper triangular with controlled
eigenvalues). At convergence (high $k$), all methods achieve comparable
residuals, confirming the algorithms are not specific to normal matrices.
Nonnormal residuals are larger at low~$k$ due to transient growth, not
algorithmic instability---increasing $k$ restores accuracy.

\begin{table}[t]
  \centering
  \caption{Timing (ms) and residuals for $d=500$, $\rho(A)=0.9$.
    Timings are medians of 10 runs. Residuals are $\|I - (I-A)S_k\|_F$.}
  \label{tab:timing}
  \begin{tabular}{@{}lrccc@{}}
    \toprule
    Method & $k$ & Time & Res (normal) & Res (nonnormal) \\
    \midrule
    Binary & 128 & 8 & $2.5 \times 10^{-6}$ & $56$ \\
    Binary & 512 & 10 & $1.1 \times 10^{-13}$ & $8.6 \times 10^{-8}$ \\
    Quinary & 125 & 8 & $3.4 \times 10^{-6}$ & $57$ \\
    Quinary & 625 & 10 & $8.6 \times 10^{-14}$ & $1.0 \times 10^{-11}$ \\
    Radix-9 & 81 & 9 & $4.0 \times 10^{-4}$ & $37$ \\
    Radix-9 & 729 & 13 & $7.4 \times 10^{-14}$ & $3.1 \times 10^{-13}$ \\
    Radix-15 & 225 & 18 & $2.3 \times 10^{-6}$ & $2.5$ \\
    Radix-15 & 3375 & 25 & $2.2 \times 10^{-6}$ & $2.4 \times 10^{-6}$ \\
    \bottomrule
  \end{tabular}
\end{table}
For exact methods, residuals depend on $k$ (series length), not the method.
The radix-15 approximate kernel converges to a floor of $\approx 10^{-6}$
set by the kernel precision ($\max|[f]_j - 1| < 9 \times 10^{-7}$).

\begin{figure}[t]
  \centering
  \includegraphics[width=0.75\linewidth]{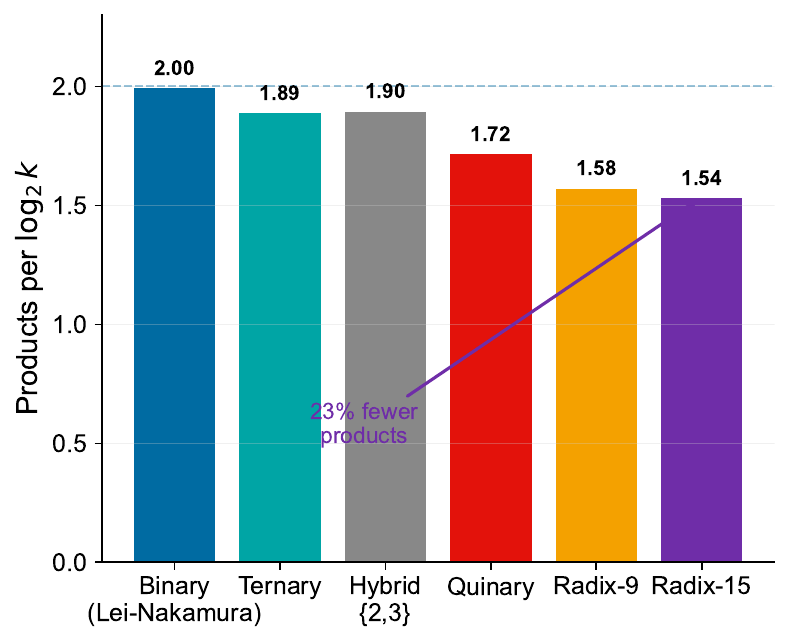}
  \caption{Asymptotic coefficients (matrix products per $\log_2 k$) for each method.
    Lower is better. Radix-15 achieves the best rate at $1.54$.}
  \label{fig:coefficients}
\end{figure}

\paragraph{Stability.}
The exact methods (binary, quinary, radix-9) achieve residuals near machine
precision for $\rho(A) \in [0.5, 0.95]$ and $\kappa(I-A)$ up to $10^4$.
Radix-15's residual floor ($\approx 10^{-6}$) is an inherent property of
the approximate kernel, not numerical instability; a higher-precision kernel
search would lower this floor.

\paragraph{Summary.}
These experiments confirm that the theoretical product savings translate to
practice: radix-9 and radix-15 achieve the same accuracy as binary with
25\% fewer products.
The exact methods reach machine precision; the approximate radix-15 kernel
trades a modest accuracy floor for the best asymptotic coefficient ($1.54$).

%% Related Work - Positioned before conclusion
%% OCAR: Opening (context), Challenge (limitations of prior work),
%%       Action (how we differ), Resolution (our advantage)

\section{Related Work}
\label{sec:related}

%% OPENING: Neumann series evaluation history
%% Sentence chain: validation → history → first algorithm → improvements → claim

Having validated our methods empirically, we now position them in context.
The problem of evaluating $S_k(A)$ efficiently spans three decades.
The first practical algorithm, by Westreich~\cite{westreich1989}, achieved
$3\log_2 k$ matrix products via binary splitting. Lei and Nakamura~\cite{lei1992}
improved this to $2\log_2 k - 1$ by sharing intermediate computations. Dimitrov
and Cooklev~\cite{dimitrov1995} introduced hybrid radix-$\{2,3\}$ methods using
double-base number representations, but claimed that higher radices offer no
further improvement.

%% CHALLENGE: The claim was wrong
%% Sentence chain: contradiction → evidence → gap

This claim was contradicted by Matsumoto, Takagi, and Takagi~\cite{matsumoto2018},
who demonstrated that higher radices \emph{do} yield more efficient algorithms.
Their analysis showed improved asymptotic constants are achievable. However, they
provided no explicit kernel circuits or analysis of which radix values are optimal.
To fill this gap, we adapt techniques from optimal polynomial evaluation---specifically, composition methods that surpass the Paterson--Stockmeyer bound.

%% ACTION: General polynomial evaluation context
%% Sentence chain: PS → optimality → Sastre → connection

For general matrix polynomials, the Paterson--Stockmeyer (PS) method~\cite{paterson1973}
achieves $O(\sqrt{n})$ products for degree-$n$ evaluation.
This bound was proved optimal within the PS evaluation scheme~\cite{fasi2019}.
However, it can be surpassed for specific polynomials using composition methods---Sastre and Ib\'a\~nez~\cite{sastre2021} achieved degree~15 in 4~products, exceeding the PS bound of degree~9 achievable with 4~products.

%% RESOLUTION: How we differ and what we add
%% Sentence chain: our contribution → what's new → significance

Our radix kernels adapt these composition ideas to the geometric series structure, extending Gustafsson et al.'s quinary construction~\cite{gustafsson2017} to higher radices.
The radix-9 kernel is the first published exact circuit beyond radix-5, and our kernel-to-update framework unifies algorithm design: better kernels automatically yield better series algorithms.

%% Conclusion - Ties back to introduction

\section{Conclusion}
\label{sec:conclusion}

%% Summary: Echo the challenge and contributions from introduction
We set out to extend Neumann series evaluation beyond the known radix-5
kernel~\cite{gustafsson2017}, where no higher-radix constructions existed.
We addressed this gap with three contributions:
\begin{enumerate}
  \item The radix-9 kernel achieves $\mu_9 = 3$ with rational
    coefficients---the first exact construction beyond radix-5---yielding
    coefficient $1.58\log_2 k$ (21\% over binary splitting).
  \item A general radix-kernel framework (Section~\ref{sec:generalradix})
    enables iterative use of approximate kernels despite spillover.
  \item The approximate radix-15 kernel achieves coefficient $1.54\log_2 k$
    within this framework.
\end{enumerate}

%% Future: Brief forward look
The transition from exact (radix-9) to approximate (radix-15) suggests
that spillover is the norm at higher radices. The general radix-kernel
framework provides a principled approach for this regime.
This motivates several directions:
\begin{enumerate}
  \item \emph{Exact vs.\ approximate boundary}: Can we prove that exact
    radix-15 kernels do not exist, or find one?
  \item \emph{Higher radices}: Can approximate kernels with $\mu_m$
    growing sublogarithmically improve the coefficient beyond $1.54$?
  \item \emph{Structured matrices}: How do these algorithms extend to
    sparse or Toeplitz settings?
\end{enumerate}

These techniques may also benefit log-determinant estimation~\cite{sao2025logdet},
where polynomial traces from radix-kernel sums could replace individual
power traces while reducing the number of matrix products.

\section*{Acknowledgements}
This work was supported by the Laboratory Directed Research and Development
Program of Oak Ridge National Laboratory, managed by UT-Battelle, LLC,
for the U.S. Department of Energy.

%% Bibliography
\paperbibliography

%% Appendix
\clearpage
\appendix
%% Appendix - Kernel Coefficients and Numerical Validation

\section{Kernel Coefficients}
\label{app:coeffs}

This appendix provides explicit formulas for the optimized kernels and
numerical validation on test matrices.

\subsection{Quinary Kernel (2 Products)}

The quinary kernel $T_5(B) = I + B + B^2 + B^3 + B^4$ is computed using
2~matrix products:
\begin{align}
  U &= B^2, \label{eq:quinary-U} \\
  V &= U(B + U) = B^3 + B^4, \label{eq:quinary-V} \\
  T_5(B) &= I + B + U + V. \label{eq:quinary-final}
\end{align}
This is exact: all 5~coefficients equal~1.

\subsection{Radix-9 Kernel (3 Products)}

The radix-9 kernel $T_9(B) = \sum_{j=0}^{8} B^j$ uses 3~products:
\begin{align}
  U &= B^2, \label{eq:radix9-U} \\
  V &= U(B + 2U) = B^3 + 2B^4, \label{eq:radix9-V} \\
  P &= (0.15B + 2U + V)(0.275B - 0.125U + 0.25V), \label{eq:radix9-P} \\
  T_9(B) &= I + B + \tfrac{767}{800}U + \tfrac{15}{32}V + P. \label{eq:radix9-final}
\end{align}
The rational coefficients in~\eqref{eq:radix9-final} are
$\frac{767}{800} = 0.95875$ and $\frac{15}{32} = 0.46875$.

\subsection{Numerical Validation}
\label{app:validation}

We validate convergence on a $64 \times 64$ symmetric positive definite (SPD) matrix~$A$ with
geometrically spaced eigenvalues in $[10^{-2}, 1]$, giving
$\kappa(I-A) \approx 10^4$. In this appendix, let $B = I - A$
(distinct from the kernel argument~$B$ in the main text).

\paragraph{Setup.}
For series index~$k$, each method computes an approximation~$S_k$ to
$(I-B)^{-1}$. We measure the relative residual
$\|I - (I-B)S_k\|_F / \sqrt{d}$, where $\|\cdot\|_F$ denotes the Frobenius norm.

\paragraph{Results.}
Table~\ref{tab:mmcounts} lists the matrix product count
required to reach machine precision ($\approx 10^{-13}$).

\begin{table}[ht]
\centering
\caption{Matrix product counts to reach $10^{-13}$ residual ($d=64$, $\kappa \approx 10^4$).}
\label{tab:mmcounts}
\begin{tabular}{@{}lcc@{}}
\toprule
Method & Products & Residual \\
\midrule
Binary ($m=2$)   & 38 & $5.6 \times 10^{-13}$ \\
Ternary ($m=3$)  & 36 & $4.7 \times 10^{-13}$ \\
Quinary ($m=5$)  & 32 & $3.9 \times 10^{-13}$ \\
Radix-9          & 32 & $4.1 \times 10^{-13}$ \\
Radix-15 (approx.)& 30 & $9.9 \times 10^{-8}$ \\
\bottomrule
\end{tabular}
\end{table}

Radix-15 uses the fewest products but converges to $\approx 10^{-7}$
rather than machine precision, due to kernel coefficient error.
The naive method requires $k-1 = 10^4$ products for comparable accuracy.

\paragraph{Efficiency comparison.}
Per radix step, the update cost is $C(m) = \mu_m + 2$ products:
\begin{itemize}
  \item Binary: $C(2) = 0 + 2 = 2$, coefficient $\frac{2}{\log_2 2} = 2.00$;
  \item Ternary: $C(3) = 1 + 2 = 3$, coefficient $\frac{3}{\log_2 3} \approx 1.89$;
  \item Quinary: $C(5) = 2 + 2 = 4$, coefficient $\frac{4}{\log_2 5} \approx 1.72$;
  \item Radix-9: $C(9) = 3 + 2 = 5$, coefficient $\frac{5}{\log_2 9} \approx 1.58$;
  \item Radix-15: $C(15) = 4 + 2 = 6$, coefficient $\frac{6}{\log_2 15} \approx 1.54$.
\end{itemize}
Radix-9 achieves the best rate among exact kernels; radix-15 achieves the best overall.

\subsection{Approximate Kernels: Residual vs Radix Trade-off}
\label{app:approx}

For a single truncated Neumann computation (no recursion), we compare
approximate kernels at fixed matrix product budgets (6 and 7 products).
Test matrix: $d=64$, geometrically spaced eigenvalues in $[0.01, 1]$,
giving $\rho(B) = 0.99$.

\paragraph{Key finding: Higher radix dominates.}
Table~\ref{tab:approx} compares kernels at 6 products and 7 products budgets.
Even with high kernel residual, higher radix gives better inverse approximation.
Error is measured in spectral norm: $\|S - A^{-1}\|_2 / \|A^{-1}\|_2$,
where $\|\cdot\|_2$ denotes the largest singular value.

\begin{table}[ht]
\centering
\caption{Inverse approximation error (spectral norm) for single kernel
application. ``vs Exact'': ratio to exact Neumann at same radix.}
\label{tab:approx}
\resizebox{\columnwidth}{!}{%
\begin{tabular}{@{}lcccc@{}}
\toprule
Method & Prods & Kern.\ Res. & Inv.\ Err. & vs Exact \\
\midrule
\multicolumn{5}{l}{\textit{6-product budget}} \\
Standard $S_7$ (radix 7)  & -- & 0 & 0.932 & -- \\
Radix-15 & 4 & $9{\times}10^{-12}$ & 0.860 & 1.000 \\
Radix-16 & 4 & $1{\times}10^{-2}$ & 0.852 & 1.000 \\
Radix-17 & 4 & $1{\times}10^{-1}$ & 0.844 & 1.001 \\
\midrule
\multicolumn{5}{l}{\textit{7-product budget}} \\
Standard $S_8$ (radix 8)  & -- & 0 & 0.923 & -- \\
Radix-24 & 5 & $7{\times}10^{-5}$ & 0.786 & 1.000 \\
Radix-27 & 5 & $1{\times}10^{-3}$ & 0.762 & 1.000 \\
Radix-31 & 5 & $1{\times}10^{-1}$ & 0.733 & 1.002 \\
Radix-33 & 5 & $6{\times}10^{-1}$ & 0.722 & 1.007 \\
\bottomrule
\end{tabular}}
\end{table}

\paragraph{Interpretation.}
The ``vs Exact'' column shows degradation from kernel approximation error.
Even radix-33 with residual 0.6 is only 0.7\% worse than an exact radix-33 kernel.
The dominant factor is radix $m$: error scales as $\rho(B)^m$, so larger $m$
gives exponentially better approximation, overwhelming the kernel residual effect.

\paragraph{Coefficients for radix-33 kernel (5 products).}
The approximate coefficients $c_j$ for $\widetilde{T}_{33}(B) = \sum_{j=0}^{32} c_j B^j$:
\begin{table}[ht]
\centering
\small
\begin{tabular}{@{}r l@{}}
\toprule
$j$ & $c_j$ values \\
\midrule
0--10 & 1.00, 1.00, 1.00, 0.98, 1.00, 1.02, 1.00, 0.69, 1.12, 0.99, 1.05 \\
11--21 & 0.88, 1.11, 0.69, 0.73, 1.01, 1.27, 1.04, 1.08, 1.06, 0.98, 1.00 \\
22--32 & 0.99, 0.88, 1.13, 0.95, 1.02, 0.97, 1.01, 0.99, 1.00, 1.00, 1.00 \\
\bottomrule
\end{tabular}
\end{table}
Coefficients range from 0.69 to 1.27, yet the inverse error is only 1\% worse
than exact. No spillover terms exist beyond degree 32.

%% Appendix - Radix-15 Details

\section{Radix-15 Kernel: Explicit Circuit}
\label{app:radix15}

The circuit~\eqref{eq:radix15kernel} with $P_1 = B^2$ is fully
specified by:
\begin{align*}
  P_2 &= (0.238I + 0.241B + 1.574P_1)
         (0.048I - 0.047B + 0.889P_1), \\
  P_3 &= (0.263I + 0.919B + 0.842P_1 + 0.819P_2) \\
      &\quad \times
         (0.184I + 0.068B + 0.134P_1 - 1.405P_2), \\
  P_4 &= (-0.005I - 1.385B + 0.022P_1
         + 0.122P_2 + 0.275P_3) \\
      &\quad \times
         (-0.701I + 0.602B - 0.624P_1
         - 0.137P_2 + 0.328P_3), \\
  \widetilde{T}_{15} &= I + 0.078B + 1.778P_1
    + 0.664P_2 - 0.024P_3 + P_4.
\end{align*}
These coefficients achieve
$\max_{0 \leq j \leq 14}|[\widetilde{T}_{15}]_j - 1|
  < 9 \times 10^{-7}$
with spillover $0.19\,B^{15} + 0.46\,B^{16} + O(B^{17})$,
giving error map leading coefficient
$c = 1 - [\widetilde{T}_{15}]_{15} \approx 0.81$.

\end{document}